\documentclass[12pt,a4paper]{article}
\pdfoutput=1
\usepackage{amsmath}
\usepackage{amsfonts}
\usepackage{amssymb}
\usepackage{amsthm}
\usepackage{graphicx}
\newtheorem{theorem}{Theorem}
\newtheorem{corollary}{Corollary}
\newtheorem{lemma}{Lemma}
\newtheorem{remark}{Remark}
\begin{document}

\begin{center}
\Large \bf Asymptotic analysis of a boundary-value problem\\ with
the nonlinear boundary multiphase interactions \\ in a perforated
domain
\end{center}

\begin{center}
\large \bf  Taras A.~Mel'nyk and Olena A.~Sivak
\end{center}

\begin{center}
\small Department of Mathematical Physics, National Taras Shevchenko University of Kyiv,\\
 01033 Kyiv, Ukraine\\
E-mail: melnyk@imath.kiev.ua, o.a.sivak@gmail.com
\end{center}

\begin{abstract} We consider a boundary-value problem for the second order elliptic differential operator with rapidly
oscillating coefficients in a domain $\Omega_{\varepsilon}$ that is $\varepsilon-$periodically perforated  by small holes. The  holes are divided into two $\varepsilon-$periodical sets depending on the boundary interaction at  their surfaces.  Therefore, two different nonlinear Robin
boundary conditions $\sigma_\varepsilon (u_\varepsilon) + \varepsilon \kappa_{m} (u_\varepsilon) = \varepsilon
g^{(m)}_\varepsilon, \ m=1, 2,$ are given on the corresponding boundaries of the small holes. The asymptotic analysis of this
problem is made as $\varepsilon\to0,$ namely the convergence theorem both for the solution and for the energy integral is proved without using extension operators, the asymptotic approximations both for the solution and for the energy integral are constructed and the corresponding
error estimates  are obtained.
\end{abstract}

\section{Introduction and statement of the problem}

In recent years,  a rich collection of new results on asymptotic analysis of boundary-value problems in perforated domains is appeared
(see for example \cite{Cioranescu-Paulin-2}-\cite{Zhikov-Rychaho}). The classical method proposed by E. Khruslov \cite{Khruslov} and D. Cioranescu and J. Saint Jean Paulin \cite{Cioranescu-Paulin-1} is based on a special bounded extension of solutions in Sobolev spaces.
It was established by V. Zhikov \cite{Zhikov-1,Zhikov-2} that the homogenization results can be obtained without using the extension technique in Sobolev spaces in periodically perforated domains.
It should be mentioned the paper \cite{Cioranescu-Donato-Zaki}, where the homogenization results for an elliptic problem with a nonlinear
boundary condition in a perforated domain  were obtained with the help of a new unfolding method that does not need any extension operators as well.

In this paper we use this simple Zhikov's approach and the scheme of the paper  \cite{Mel-m2as-08},
where the full asymptotic analysis (the convergence of  the solution and the energy integral, the approximation for the solution and
the corresponding asymptotic error estimate in the Sobolev space $H^1)$ was made for  an elliptic problem with a nonlinear
boundary condition in a thick junction.

Let $B$ be a finite union of smooth disjoint nontangent domains strictly lying in the unit square
$
\square:= \{ \xi \in \mathbb{R}^n: \ \ 0<\xi_i<1,\quad
i=\overline{1,n} \}.
$
In an arbitrary way, we divide $B$ into two sets,
$B^{(1)}=\bigcup\limits_{k=1}^{N_1}B^{(1)}_{k}$ and
$B^{(2)}=\bigcup\limits_{k=1}^{N_2}B^{(2)}_{k}$.
Let us introduce the following notations:
$$
Q_0 := \square \setminus \overline{B},\quad {\cal
B}^{(m)}:=\bigcup\nolimits_{ z \in \mathbb{Z}^n} \bigl( z +
B^{(m)}\bigr),  \quad {\cal B}^{(m)}_\varepsilon :=\varepsilon
{\cal B}^{(m)}=\{ x \in \mathbb{R}^n: \ \varepsilon^{-1} x \in
{\cal B}^{(m)} \}, 
$$$$
m = 1, 2,
$$
where $\varepsilon$ is a small parameter. Let $\Omega$ be a smooth
bounded domain in $\mathbb{R}^n.$ Define the following perforated
domain
$
\Omega_\varepsilon=\Omega \setminus \overline{\bigl({\cal
B}^{(1)}_\varepsilon \cup {\cal B}^{(2)}_\varepsilon\bigr)}
$
and require the domain $\Omega_\varepsilon$ to be a domain with
the Lipschitz boundary. Denote $\Gamma_\varepsilon=\partial \Omega
\cap \overline{\Omega _\varepsilon}$ and $\Xi^{(m)}_\varepsilon=
\Omega \cap \partial {\cal B}^{(m)}_\varepsilon, \ \ m=1, 2,$ \ \
$\Xi_\varepsilon=\Xi^{(1)}_\varepsilon\cup\Xi^{(2)}_\varepsilon.$
(see Fig.~1).

\begin{figure}[htbp]
\begin{center}
\includegraphics[width=8cm]{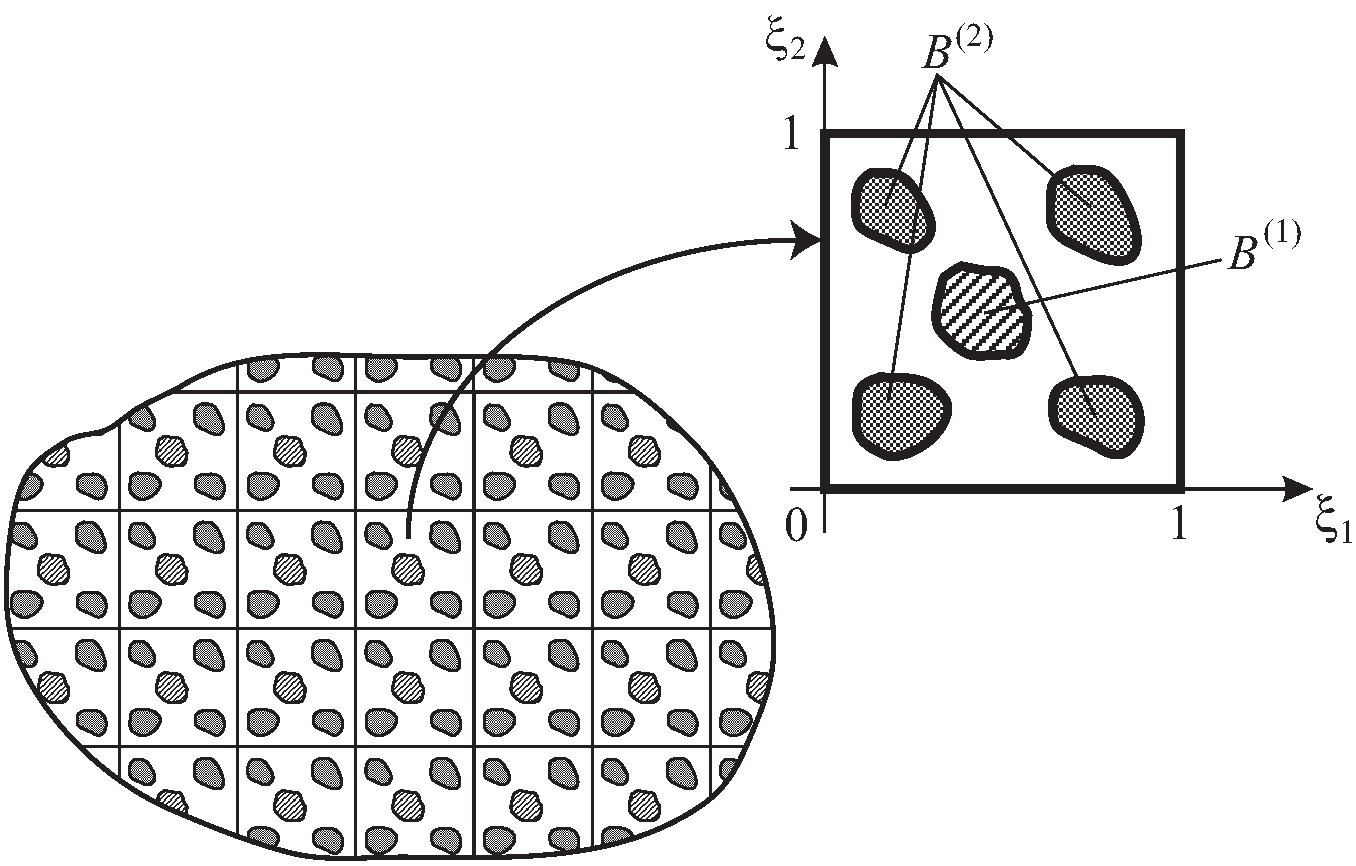} %[width=6cm]
\caption{}
\end{center}
\end{figure}

Let  $a_{ij}(\xi)$, $\xi \in \mathbb{R}^n$, $i,j=\overline{1,n},$
be smooth $1-$periodic functions such that
\begin{gather}
1) \ \ \forall i,j=1, \dots , n, \quad \forall \xi \in
\mathbb{R}^n :
\quad  a_{ij}(\xi)=a_{ji}(\xi), \notag \\
2)\ \ \exists \,\varkappa_1>0 \ \ \exists\, \varkappa_2>0  \quad
\forall\,\xi \in \mathbb{R}^n \ \ \forall \, \eta \in \mathbb{R}^n
:\quad \varkappa_1|\eta|^2 \leq  a_{ij}(\xi) \eta_i \eta_j \leq
\varkappa_2|\eta|^2.    \label{elliptic-cond}
\end{gather}

\begin{remark} Here and in the sequel we
adopt the Einstein convention of summation over repeated indexes.
\end{remark}

Let $f_\varepsilon, f_0, g^{(m)}_\varepsilon, g^{(m)}_0$ be given
functions such that $ f_\varepsilon, f_0 \in L^2(\Omega),$
$g^{(m)}_\varepsilon, g^{(m)}_0 \in H^1_0(\Omega)$ and
\begin{equation}\label{cond-1}
f_\varepsilon \stackrel{s}{\longrightarrow} f_0 \ \ \ \textrm{in} \ \ \
L^2(\Omega), \qquad g^{(m)}_\varepsilon
\stackrel{w}{\longrightarrow} g^{(m)}_0\ \ \ \textrm{weakly in} \ \ \
H^1(\Omega), \ \ m=1, 2.
\end{equation}
The given functions $\kappa_m : \mathbb{R} \to \mathbb{R}, \ m=1,
2,$ are Lipschitz continuous (it is equivalent that $\kappa_m\in
W^{1,\infty}_{loc}(\mathbb R)$) and such that
\begin{equation} \label{cond-2}
\exists \,c_1>0  \ \ \exists\, c_2>0 : \ \   c_1\le \kappa'_m \le
c_2 \quad \text{a.e. in} \ \mathbb{R} \quad (m=1, 2).
\end{equation}

In the perforated domain $\Omega_\varepsilon$ we consider the
following nonlinear problem
\begin{equation}\label{bvpipd}
\left\{
   \begin{array}{rcll}
-{\cal L}_\varepsilon(u_\varepsilon)&=&f_\varepsilon & \textrm{in}
\ \ \Omega_\varepsilon,
\\[1mm]
\sigma_\varepsilon (u_\varepsilon)+\varepsilon \kappa_1
(u_\varepsilon)&=& \varepsilon g^{(1)}_\varepsilon & \textrm{on} \
\ \Xi^{(1)}_\varepsilon,
\\[1mm]
\sigma_\varepsilon (u_\varepsilon)+\varepsilon \kappa_2
(u_\varepsilon)&=& \varepsilon g^{(2)}_\varepsilon & \textrm{on} \
\ \Xi^{(2)}_\varepsilon,
\\[1mm]
u_\varepsilon &=&0&\textrm{on}\  \ \Gamma_\varepsilon,
   \end{array}
\right.
\end{equation}
where $ {\cal L}_\varepsilon(u_\varepsilon) \equiv \partial_{x_i}
\left( a_{ij}^{\varepsilon} (x) \partial_{x_j} u_\varepsilon (x)
\right),$ \ \  $\sigma_\varepsilon (u_\varepsilon) \equiv
a_{ij}^{\varepsilon} (x) \partial_{x_j} u_\varepsilon (x)\, \nu_i,
$ \ \ $ a_{ij}^{\varepsilon} (x) \equiv a_{ij} \left(
\frac{x}{\varepsilon} \right),$ \ $\partial_{x_i} u =
\frac{\partial u}{\partial x_i}, $ \ $\left(\nu_1\bigl(
\frac{x}{\varepsilon} \bigr), \dots \nu_n\bigl(
\frac{x}{\varepsilon} \bigr)\right)$ is the outward normal.

Recall that a function $u_\varepsilon$ from the Sobolev space
$H^1(\Omega_\varepsilon, \Gamma_\varepsilon)=\{ u \in
H^1(\Omega_\varepsilon): \  u |_{\Gamma_\varepsilon} =0 \}$ is a
weak solution to problem (\ref{bvpipd}) if the following integral
identity
\begin{equation}
\label{inttot} \int_{\Omega_\varepsilon} a_{ij}^\varepsilon \,
\partial_{x_j}u_\varepsilon \, \partial_{x_i} \varphi\,dx
+\varepsilon\sum_{m=1}^2\int_{\Xi^{(m)}_\varepsilon}
\kappa_m(u_\varepsilon) \varphi\,ds_x =\int_{\Omega_\varepsilon}
f_\varepsilon \varphi\, dx +\varepsilon
\sum_{m=1}^2\int_{\Xi^{(m)}_\varepsilon} g^{(m)}_\varepsilon
\varphi\, ds_x
\end{equation}
holds for any function $\varphi \in H^1(\Omega_\varepsilon,
\Gamma_\varepsilon).$

\medskip

Our goal is to study the asymptotic behavior of $u_\varepsilon$ as $\varepsilon\to0.$  Also it will be understandable further how conduct  research in the case of  $p$-multiphase interactions in perforated domains.

%%%%%%%%%%%%%%%%%%%%%%%%%%%%%%%%%%%%%%%%%%%%%%%%
\section{Auxiliary uniform estimates}
%%%%%%%%%%%%%%%%%%%%%%%%%%%%%%%%%%%%%%%%%%%%%%

Let $H^1_{{\rm per}}(Q_0)=\{v \in H^1(Q_0):\  v-
\textrm{1-periodic in} \ \ \xi_1,\dots ,\xi_n\}.$ Obviously, we
can periodically extend  every function $v$ from $H^1_{{\rm
per}}(Q_0)$ into $H^1_{loc}\bigl(\mathbb{R}^n \setminus
\overline{({\cal B}^{(1)}\cup {\cal B}^{(2)})}\,\bigr);$ this
extension will be denoted again by $v.$
Let $\psi^{(m)}_0 \in H^1_{{\rm per}}(Q_0), \ m=1, 2,$ be  weak
solutions to the corresponding problems
\begin{equation}
\label{kznakomper} \left\{
\begin{array}{l}
{\cal L}_{\xi \xi} \bigl(\psi^{(1)}_0\bigr)=q_1 \  \  \textrm{in}
\ \ Q_0,
\\[1mm]
\sigma_\xi\bigl(\psi^{(1)}_0\bigr)=1 \ \   \textrm{on} \  \
S^{(1)},
\\[1mm]
\sigma_\xi\bigl(\psi^{(1)}_0\bigr)=0 \ \
\textrm{on} \  \ S^{(2)},
\\[1mm]
\langle \psi^{(1)}_0 \rangle_{Q_0}=0,
\end{array}
\right. \quad \qquad \left\{
\begin{array}{l}
{\cal L}_{\xi \xi} \bigl(\psi^{(2)}_0\bigr)=q_2 \  \  \textrm{in}
\ \ Q_0,
\\[1mm]
\sigma_\xi\bigl(\psi^{(2)}_0\bigr)=0 \ \   \textrm{on} \  \
S^{(1)},
\\[1mm]
\sigma_\xi\bigl(\psi^{(2)}_0\bigr)=1 \ \ \textrm{on} \  \ S^{(2)},
\\[1mm]
\langle \psi^{(2)}_0 \rangle_{Q_0}=0,
\end{array}
\right.
\end{equation}
where ${\cal L}_{\xi \xi} (\psi) \equiv
\partial_{\xi_i} \left( a_{ij} (\xi)
\partial_{\xi_j} \psi (\xi)\right),$  \ $\sigma_\xi (\psi) \equiv a_{ij} \partial_{\xi_j}\psi (\xi)
\nu_i (\xi),$ \ $\left( \nu_1,\dots,\nu_n\right)$ is the outward
normal to $S,$ \ $S=S^{(1)}\cup S^{(2)},$ \ $S^{(m)}=\partial
B^{(m)},$ \ $q_m=\frac{|S^{(m)}|}{|Q_0|},$
\ $|S^{(m)}|= \text{meas}_{\,2}\, S^{(m)}$  \ $(m=1,2),$ \ $|Q_0|= \text{meas}_{\, 3}\, Q_0,$
\ $ \langle \psi \rangle_{Q_0}=\int_{Q_0} \psi (\xi)\, d\xi.$
\ The existence and uniqueness of the solutions to problems
(\ref{kznakomper}) follows from the lemma.

\begin{lemma}\label{cell-lemma}
 Let $F_i \in L^2(Q_0), i=\overline{0,n}, \ F^{(m)}_{n+1}\in L^2(S^{(m)}), \ m=1, 2.$
\ There exists a  solution $N \in H^1_{{\rm per}}(Q_0)$ to the
following problem
\begin{equation}
\label{kznakomper_zag} \left\{
\begin{array}{l}
- {\cal L}_{\xi \xi}(N) = F_0 + \partial_{\xi_i} F_i \ \
\textrm{in} \ \ \ Q_0,
\\[1mm]
\sigma_\xi(N)= -F_i \nu_i + F^{(1)}_{n+1} \ \ \ \textrm{on} \ \ \
S^{(1)},
\\[1mm]
\sigma_\xi(N)= -F_i \nu_i + F^{(2)}_{n+1} \ \ \ \textrm{on} \ \ \
S^{(2)},
\end{array}
\right.
\end{equation}
if and only if
\begin{equation}
\langle F_0 \rangle_{Q_0}+\langle F^{(1)}_{n+1} \rangle_{S^{(1)}}
+ \langle F^{(2)}_{n+1} \rangle_{S^{(2)}} =0.
\end{equation}
In addition this solution is defined up to an additive constant.
\end{lemma}

The proof is standard (see for instance \cite{OYS}). Then the
$\varepsilon$-periodic functions $\psi^{(m)}_0 \left(
\frac{x}{\varepsilon} \right), \ x \in \Omega_\varepsilon, \ m=1,
2,$ satisfy the following relations
$$
\left\{
\begin{array}{l}
\frac{\partial}{\partial x_i} \bigl( a_{ij}^\varepsilon (x)
\frac{\partial}{\partial x_j} \bigl( \psi^{(1)}_0 \left(
\frac{x}{\varepsilon} \right) \bigr) \bigr) = \varepsilon^{-2} q_1
\ \ \textrm{in}\  \Omega_\varepsilon,
\\
\sigma_\varepsilon \bigl( \psi^{(1)}_0 \left(
\frac{x}{\varepsilon} \right) \bigr)= \varepsilon^{-1} \ \
\textrm{on}\ \Xi^{(1)}_\varepsilon,
\\
\sigma_\varepsilon \bigl( \psi^{(1)}_0 \left(
\frac{x}{\varepsilon} \right) \bigr)= 0 \ \ \textrm{on}\
\Xi^{(2)}_\varepsilon,
\end{array}
\right.
\qquad
\left\{
\begin{array}{l}
\frac{\partial}{\partial x_i} \bigl( a_{ij}^\varepsilon (x)
\frac{\partial}{\partial x_j} \bigl( \psi^{(2)}_0 \left(
\frac{x}{\varepsilon} \right) \bigr) \bigr) = \varepsilon^{-2}q_2
\ \ \textrm{in}\  \Omega_\varepsilon,
\\
\sigma_\varepsilon \bigl( \psi^{(2)}_0 \left(
\frac{x}{\varepsilon} \right) \bigr)= 0 \ \ \textrm{on}\
\Xi^{(1)}_\varepsilon,
\\
\sigma_\varepsilon \bigl( \psi^{(1)}_0 \left(
\frac{x}{\varepsilon} \right) \bigr)= \varepsilon^{-1} \ \
\textrm{on}\ \Xi^{(2)}_\varepsilon.
\end{array}
\right.
$$
Multiplying with arbitrary function $\varphi \in
H^1(\Omega_\varepsilon, \Gamma_\varepsilon)$ the corresponding
differential equation, integrating over $\Omega_\varepsilon$  and
taking into account the boundary conditions, we get the following
integral identities
\begin{equation}\label{identities}
\varepsilon \int_{\Xi^{(m)}_\varepsilon} \varphi \, ds_x =
\varepsilon \int_{\Omega_\varepsilon} a_{ij}^\varepsilon(x) \,
\partial_{\xi_j} \psi^{(m)}_0(\xi)|_{\xi=\frac{x}{\varepsilon}} \,
\partial_{x_i} \varphi \,dx + q_m \int_{\Omega_\varepsilon}
\varphi \, dx \quad  m=1, 2.
\end{equation}

Due to the regularity properties of solutions to elliptic
problems we have
\begin{equation}\label{est-0}
\sup_{x\in \Omega_\varepsilon} \bigl|\nabla_{\xi}\psi^{(m)}_0(\xi)
|_{\xi=\frac{x}{\varepsilon}}\bigr|= \sup_{\xi \in Q_0}
\bigl|\nabla_{\xi}\psi^{(m)}_0(\xi)\bigr|  \le C_0 \quad (m=1, 2).
\end{equation}
Using Cauchy's inequality with $\delta$ $(ab\le\delta
a^{2}+\frac{b^{2}}{4\delta},\ \ a, b, \delta>0)$ and
(\ref{est-0}), we deduce from (\ref{identities}) the following
estimates (m=1, 2)
\begin{gather}\label{est-1}
\varepsilon \int_{\Xi^{(m)}_\varepsilon} \varphi^2 \, ds_x
 \le
C_1 \left( \varepsilon^2 \int_{\Omega_{\varepsilon}}|\nabla_{x}
\varphi |^2 \, dx + \int_{\Omega_{\varepsilon}} \varphi^2 \, dx
\right),
\\
\int_{\Omega_{\varepsilon}} \varphi^2 \, dx \le C_2 \left(
\varepsilon^2 \int_{\Omega_{\varepsilon}}|\nabla_{x} \varphi |^2
\, dx +
 \varepsilon \int_{\Xi^{(m)}_\varepsilon} \varphi^2 \, ds_x \right)
\qquad \forall\, \varphi \in H^1(\Omega_{\varepsilon},
\Gamma_\varepsilon), \label{est-2}
\end{gather}
where the constant $C_1$ and $C_2$ are independent of
$\varepsilon.$

\begin{remark}
In  what follows all constants $\{C_{i}\}$ and $\{c_{i}\}$ in
inequalities are independent of the parameter $\varepsilon.$
\end{remark}

It follows from (\ref{est-1}) and (\ref{cond-1}) that
\begin{equation}\label{est-3}
\sqrt{\varepsilon} \sum\nolimits_{m=1}^2
\|g^{(m)}_\varepsilon\|_{L^2(\Xi^{(m)}_\varepsilon)} \le C_3.
\end{equation}
Also with the help of (\ref{est-1}) and (\ref{est-2}) it is easy
to prove that the usual norm $\|\cdot\|_{H^{1}(\Omega_{\varepsilon
})}$ is uniformly equivalent with respect to $\varepsilon$  to a
new norm
\begin{equation*}
\|u\|_\varepsilon:=\Bigl(\int_{\Omega_{\varepsilon}}|\nabla
u|^2\,dx + \varepsilon \int_{\Xi_\varepsilon}u^2 \,
ds_x\Bigr)^{1/2}
\end{equation*}
in the space $H^1(\Omega_\varepsilon, \Gamma_\varepsilon),$ i.e.,
there exist constants $C_{3}>0,$ $C_4 > 0$ and $\varepsilon_{0}>0$
such that for any $\varepsilon \in (0,\varepsilon _{0})$ and $u
\in H^1(\Omega _\varepsilon, \Gamma_\varepsilon)$ the following
relations hold
\begin{equation}\label{equiv-norms}
C_3\|u\|_{H^{1}(\Omega_{\varepsilon })}\leq\|u\|_{\varepsilon }
\le C_4\|u\|_{H^{1}(\Omega_{\varepsilon })}.
\end{equation}

%-----------------------------------------------------
\subsection{Existence and uniqueness of the solution to problem
(\ref{bvpipd})}
%------------------------------------------------

Associated with (\ref{bvpipd}), we consider the energy functional
\begin{equation}\label{energy functional}
I_\varepsilon[u]:= \tfrac{1}{2} \int\limits_{\Omega_\varepsilon}
a_{ij}^\varepsilon (x)\, \partial_{x_i}u \,
\partial_{x_j}u\, dx +\varepsilon \sum_{m=1}^2 \Bigl(
\int\limits_{\Xi^{(m)}_\varepsilon} K^{(m)}(u)\, ds_x -
\int\limits_{\Xi^{(m)}_\varepsilon} g^{(m)}_\varepsilon u\, ds_x\Bigr)
-\int\limits_{\Omega_\varepsilon} f_\varepsilon u\, dx
\end{equation}
on $H^1(\Omega_\varepsilon, \Gamma_\varepsilon),$ where
\begin{equation}\label{}
K^{(m)}(z)=\int^{z}_{0}\kappa_m(t)\,dt\quad \forall\, z\in \mathbb R,
\quad m=1, 2.
\end{equation}

It is easy to prove that if $u_\varepsilon$ is a minimizer of
$I_\varepsilon$ at a fixed value of $\varepsilon,$ then
$u_\varepsilon$ is a weak solution to problem (\ref{bvpipd}).

\begin{theorem}
 At each fixed value of $\varepsilon$
problem (\ref{bvpipd}) has exactly one solution $u_\varepsilon\in
H^1(\Omega_\varepsilon, \Gamma_\varepsilon)$ for which the
following estimate
\begin{equation}\label{uniform estimate}
  \|u_\varepsilon\|_{H^1(\Omega_\varepsilon)} \le
  C_1\Bigl(1 + \|f_\varepsilon\|_{L^2(\Omega_\varepsilon)} +
\sqrt{\varepsilon} \sum\nolimits_{m=1}^2
\|g^{(m)}_\varepsilon\|_{L^2(\Xi^{(m)}_\varepsilon)}\Bigr) \le C_2
\end{equation}
holds, where the constants $C_1$ and $C_2$ are independent of
$\varepsilon,$ $f_\varepsilon, g^{(m)}_\varepsilon$ and
$u_\varepsilon.$
\end{theorem}

\begin{proof}
 Integrating inequalities in (\ref{cond-2}), we obtain
\begin{equation}\label{est-4}
c_1 t^2 + \kappa_m(0)\,t \le \kappa_m(t)\, t \le c_2 t^2 +
\kappa_m(0)\, t \quad \forall\, t\in\mathbb R,
\end{equation}
whence it follows that
\begin{equation}\label{est-5}
\frac{c_1}{2} \, z^{2} +  \kappa_m(0) z \le  K^{(m)}(z)\le
\frac{c_2}{2} \, z^{2} + \kappa_m(0)z \quad \forall \, z\in \mathbb R
\quad m=1, 2.
\end{equation}

 Using (\ref{equiv-norms}), (\ref{est-4}), (\ref{est-5}),
(\ref{cond-1}) and the same arguments as in Theorem~1
(\cite{Mel-m2as-08}), we can prove the coercitivity condition on
$I,$ i.e., the following inequality
\begin{equation}\label{coercitiv-cond}
I_\varepsilon[u]\ge C_1 \|u\|^2_{H^{1}(\Omega_{\varepsilon})} -
C_2
\end{equation}
holds for any function $u\in H^{1}(\Omega_\varepsilon,
\Gamma_\varepsilon).$

With the help of (\ref{identities}) we can re-write the energy
functional as
\begin{multline*}
I_\varepsilon[u]= \frac{1}{2} \int_{\Omega_\varepsilon}
a_{ij}^\varepsilon (x)\, \partial_{x_i}u \,
\partial_{x_j}u\, dx
+ \sum_{m=1}^2\Bigl( \varepsilon \int_{\Omega_\varepsilon} a_{ij}^\varepsilon
\, \partial_{\xi_j}\bigl(\psi_0
(\xi)\bigr)|_{\xi=\frac{x}{\varepsilon}}\, \kappa_m (u) \,
\partial_{x_i}u\, dx +
\\
+ q_m \int\limits_{\Omega_\varepsilon} K^{(m)}(u)\, dx
 - \varepsilon \int\limits_{\Omega_\varepsilon} a_{ij}^\varepsilon
\partial_{\xi_j}\bigl(\psi_0 (\xi)\bigr)|_{\xi=\frac{x}{\varepsilon}}
\bigl( u \partial_{x_i}g^{(m)}_\varepsilon + g^{(m)}_\varepsilon
\partial_{x_i}u \bigr) \, dx - 
\\
- q_m \int\limits_{\Omega_\varepsilon} g^{(m)}_\varepsilon u\, dx\Bigr) - 
\int\limits_{\Omega_\varepsilon} f_\varepsilon u\, dx.
\end{multline*}
Consider the function
\begin{multline*}
L(p,t,x)=\frac{1}{2} a_{ij}^\varepsilon p_i p_j+
\sum_{m=1}^2\Bigl( \varepsilon a_{ij}^\varepsilon \,
\partial_{\xi_j}\bigl(\psi_0(\xi)\bigr)|_{\xi=\frac{x}{\varepsilon}}\,
\kappa_m(t) \, p_i + q_m  K^{(m)}(t) -
\\
- \varepsilon  a_{ij}^\varepsilon
\partial_{\xi_j}\bigl(\psi_0 (\xi)\bigr)|_{\xi=\frac{x}{\varepsilon}}
\bigl( t \partial_{x_i}g^{(m)}_\varepsilon + g^{(m)}_\varepsilon
p_i \bigr) - q_m t g^{(m)}_\varepsilon \Bigr) - f_\varepsilon t.
\end{multline*}
Since
$$
\partial^2_{p_i p_j}L(p,t,x)\,  \eta_i \eta_j=
2^{-1}a_{ij}^\varepsilon(x) \eta_i \eta_j \ge \varkappa_1 |\eta|^2
\quad \forall\, p, \eta \in \mathbb{R}^n,\ \ x\in
\Omega_\varepsilon,
 $$
the function $L$ is uniformly convex in $p$ for each $x\in
\Omega_\varepsilon.$ This means that $I[\cdot]$ is weakly lower
semicontinuous on $H^1(\Omega_\varepsilon, \Gamma_\varepsilon)$
and there exists at least one minimizer (see
\cite[Chapter~8.2]{Evans}).

Thanks to (\ref{cond-2}) it is easy to prove the uniqueness of
this minimizer (see  Theorem~1~(\cite{Mel-m2as-08})).

Finally, let us deduce the uniform estimate (\ref{uniform
estimate}). Denote by $u_\varepsilon$ the solution to problem
(\ref{bvpipd}). Setting $\varphi=u_\varepsilon$ in (\ref{inttot})
and taking into account (\ref{elliptic-cond}) and the left
inequality in (\ref{est-4}), we get
\begin{equation*}
\varkappa_1 \int_{\Omega_\varepsilon}|\nabla
u_\varepsilon|^{2}\,dx + \varepsilon\,
c_1\int_{\Xi_\varepsilon}u_\varepsilon^{2}\,ds_{x} + \varepsilon
\kappa_m(0) \sum_{m=1}^2\int_{\Xi^{(m)}_\varepsilon}
u_\varepsilon\,ds_x
 \le
\int_{\Omega_\varepsilon} f_\varepsilon u_\varepsilon\, dx
+\varepsilon \sum_{m=1}^2\int_{\Xi^{(m)}_\varepsilon}
g^{(m)}_\varepsilon u_\varepsilon\, ds_x
\end{equation*}
from which
\begin{multline*}
c_2\Bigl(\int_{\Omega_\varepsilon}|\nabla u_\varepsilon|^{2}\,dx +
\varepsilon\int_{\Xi_\varepsilon}u^{2}_\varepsilon\,d\sigma_{x}\Bigr)
\le 
c_3 \sqrt{\varepsilon}\|u_\varepsilon\|_{L^2(\Xi_\varepsilon)} +  
\|f_\varepsilon\|_{L^2(\Omega_\varepsilon)}
\|u_\varepsilon\|_{L^2(\Omega_\varepsilon)} + 
\\ 
+ \varepsilon
\sum_{m=1}^2 \|g^{(m)}_\varepsilon\|_{L^2(\Xi^{(m)}_\varepsilon)}
\|u_\varepsilon\|_{L^2(\Xi^{(m)}_\varepsilon)}.
\end{multline*}
Using (\ref{equiv-norms}) and (\ref{est-1}), we derive the first
part of the estimate (\ref{uniform estimate}) from the last
inequality, and then the second one on the basis of (\ref{cond-1})
and (\ref{est-3}). \qquad \end{proof}

%%%%%%%%%%%%%%%%%%%%%%%%%%%%%%%%%%%
\section{Convergence theorem}
%%%%%%%%%%%%%%%%%%%%%%%%%%%%%%%%%%%%%%

In the sequel, $\widetilde{y}$ denotes the zero-extension of a
function $y$ defined on $\Omega_\varepsilon$ into the domain
$\Omega.$ Also we introduce the following characteristic function
\begin{equation}\label{charach-function}
\chi_{Q_0}(\xi) =
  \begin{cases}
    1, & x \in Q_0,
    \\
0, & x\in \square \setminus Q_0.
  \end{cases}
\end{equation}
It is known that $\chi_{Q_0}^\varepsilon(x):=\chi_{Q_0}(\frac{x}{\varepsilon})
\stackrel{w}{\longrightarrow} |Q_0|$\ weakly in $L^2(\Omega)$ as
$\varepsilon\to 0.$

\begin{lemma}\label{convergence-lemma}
Let $\{v_\varepsilon\}_{\varepsilon>0}$ be a sequence in
$H^1(\Omega_\varepsilon, \Gamma_\varepsilon)$ uniformly bounded in
$\varepsilon$ in $H^1(\Omega_\varepsilon, \Gamma_\varepsilon)$ and
such that
\[
\widetilde{\kappa_m(v_\varepsilon)} \to \zeta \quad \mbox{weakly
in}\ \ L^2(\Omega)\quad \mbox{as} \quad\varepsilon\to0 \quad
\quad (m=1, 2).
\]
Then for any function $\varphi \in
  H^1(\Omega_\varepsilon, \Gamma_\varepsilon)$
\begin{equation}\label{special convergence-1}
  \varepsilon\int_{\Xi^{(m)}_\varepsilon}
  \kappa_m(v_\varepsilon)\,\varphi\, ds_x \,\to\,
  q_m \int_{\Omega} \zeta(x)\, \varphi(x)\, dx \quad \mbox{as} \quad
  \varepsilon \to 0\quad  \quad (m=1, 2).
\end{equation}
\end{lemma}
\begin{proof}
By virtue of (\ref{identities}) we have
\begin{multline*}
\varepsilon\int_{\Xi^{(m)}_\varepsilon}
  \kappa(v_\varepsilon)\,\varphi\, ds_x=
\varepsilon \int_{\Omega_\varepsilon} a_{ij}^\varepsilon(x) \,
\partial_{\xi_j} \psi^{(m)}_0(\xi)|_{\xi=\frac{x}{\varepsilon}} \,
\bigl( \kappa'(v_\varepsilon)\,\partial_{x_i}v_\varepsilon\,
\varphi + \kappa(v_\varepsilon)\,\partial_{x_i}\varphi\bigr) \,dx
+ \\
+ q_m \int_{\Omega} \widetilde{\kappa(v_\varepsilon)} \,\varphi\,
\, dx, \quad  m=1, 2.
\end{multline*}
Thanks to  the Lemma's condition, (\ref{cond-2}) and
(\ref{est-0}), the first summand vanishes and the second one tends
to $q_m \int_{\Omega} \zeta(x)\, \varphi \,dx$
  as $\varepsilon\to0$ \ $m=1, 2.$
\end{proof}

\begin{remark}
From Lemma~\ref{convergence-lemma}  it follows that for any
sequence $\{v_\varepsilon\}_{\varepsilon>0}\in
H^1(\Omega_\varepsilon, \Gamma_\varepsilon),$ which is uniformly
bounded with respect to $\varepsilon,$  there exists a subsequence
$\{\varepsilon'\}\subset \{\varepsilon\}$ (again denoted by
$\{\varepsilon\})$ and a function $\zeta\in L^2(\Omega)$ such that
the convergences (\ref{special convergence-1}) hold.
\end{remark}

Using (\ref{cond-1}),  we can prove similarly as in
Lemma~\ref{convergence-lemma} that for any function $\varphi \in
H^1(\Omega_\varepsilon, \Gamma_\varepsilon)$
\begin{equation}\label{special convergence-2}
  \varepsilon\int_{\Xi^{(m)}_\varepsilon}
g^{(m)}_\varepsilon  \varphi\, ds_x \,\to\,
  |S^{(m)}| \int_{\Omega} g^{(m)}_0(x) \,  \varphi(x)\, dx \quad \mbox{as} \quad
  \varepsilon \to 0\quad
\quad (m=1, 2).
\end{equation}

Consider $1-$periodic solutions $T_l, \ l=1,\ldots,n,$ to the
following problems
\begin{equation} \label{cell-problem} \left\{
\begin{array}{l}
{\cal L}_{\xi \xi} \bigl( T_l \bigr) = - \partial_{\xi_i}a_{il} \
\ \textrm{in} \ \ Q_0,
\\[1mm]
\sigma_\xi\bigl(T_l\bigr)= - a_{il}\, \nu_i  \ \ \textrm{on} \ \
S,
\quad \langle T_l \rangle_{Q_0}=0.
\end{array}
\right.
\end{equation}
From Lemma~\ref{cell-lemma} it follows the existence and
uniqueness of  the solutions to these problems.

With the help of $T_l, \ l=1,\ldots,n,$ we define the coefficients
of the homogenized matrix $\{\widehat{a}_{ij}\}$ by the formula
\begin{equation}\label{hom-matrix}
\widehat{a}_{ij}= \langle a_{ij} + a_{ik}\partial_{\xi_k}T_j
\rangle_{Q_0}, \quad i,j\in \{1, 2, \ldots, n\}.
\end{equation}
It is easy to see that
\begin{equation}\label{hom-matrix-1}
\widehat{a}_{ij}= \langle a_{kl}\, \partial_{\xi_k}(\xi_i + T_i)
\, \partial_{\xi_l}(\xi_j + T_j)\rangle_{Q_0}
\end{equation}
i.e.,  the matrix $\{\widehat{a}_{ij}\}$ is symmetric and it is
well known that it is elliptic (see for instance \cite{OYS}).

\begin{theorem} For the solution $u_\varepsilon$ to problem (\ref{bvpipd})
there exists the following convergences
\begin{equation}\label{main-convergences}
\left.
\begin{array}{rcll}
\widetilde{u_\varepsilon} &\stackrel{w}{\longrightarrow}& \,
|Q_0|\,v_0
    &\mbox{weakly in}\quad L^2(\Omega),
\\[2mm]
\widetilde{a_{ij}^\varepsilon\, \partial_{x_j}u_\varepsilon}
&\stackrel{w}{\longrightarrow}& \, \widehat{a}_{ij}
\partial_{x_j}v_0
    &\mbox{weakly in}\quad L^2(\Omega), \quad i=1,\ldots,n,
\end{array} \right\}
\quad \text{as} \quad \varepsilon \to 0,
\end{equation}
where $v_0$ is a unique weak solution to the following
problem
\begin{equation}
\label{limit-problem} \left\{
   \begin{array}{rcll}
- \widehat{a}_{ij}\,\partial^2_{x_i x_j} v_0(x) +
\sum\limits_{m=1}^2 |S^{(m)}|\, \kappa_m(v_0(x)) &=&
\sum\limits_{m=1}^2 |S^{(m)}|\, g^{(m)}_0(x) + |Q_0|\, f_0(x),&
x\in\Omega,
\\[2mm]
v_0(x)&=&0,& x\in \partial\Omega,
   \end{array}
\right.
\end{equation}
which is called {\it homogenized problem} for (\ref{bvpipd}).

Furthermore, the following energy convergence holds as
$\varepsilon \to 0:$
\begin{multline}\label{energy}
E_\varepsilon(u_\varepsilon):= \int_{\Omega_\varepsilon}
a^\varepsilon_{ij}(x)\,\partial_{x_j}u_\varepsilon
\,\partial_{x_i}u_\varepsilon\,dx + \varepsilon
\sum\limits_{m=1}^2 \int_{\Xi^{(m)}_\varepsilon}
\kappa_m(u_\varepsilon)\,u_\varepsilon \,ds_{x}\, \to
\\
\int_{\Omega} \widehat{a}_{ij}\,\partial_{x_j}v_0
\,\partial_{x_i}v_0 \, dx\,+\, \sum\limits_{m=1}^2 |S^{(m)}|
\int_{\Omega}\kappa_m(v_0)\, v_0\,dx =: E_0(v_0).
\end{multline}

\end{theorem}

\begin{proof}
{\bf 1.} It follows from (\ref{uniform estimate}) and
(\ref{cond-2})  that the values
\[
\|\widetilde{u_\varepsilon} \|_{L^2(\Omega)},\quad \|
\widetilde{a_{ij}^\varepsilon\,
\partial_{x_j}u_\varepsilon}\|_{L^2(\Omega)}, \ \
i=1,\ldots,n,\quad
\|\widetilde{\kappa_m(u_\varepsilon)}\|_{L^2(\Omega)}, \ \ m=1, 2,
\]
are uniformly bounded with respect to $\varepsilon.$ Hence
there exists a subsequence $\{\varepsilon '\}\subset\{\varepsilon \},$
again denoted by $\{\varepsilon\},$ such that
\begin{equation}\label{3.0}
\left.
\begin{array}{rcll}
\widetilde{u_\varepsilon} &\stackrel{w}{\longrightarrow}& \,
|Q_0|\, v_0    &\mbox{weakly in}\quad L^2(\Omega),
\\[2mm]
\widetilde{a_{ij}^\varepsilon\, \partial_{x_j}u_\varepsilon}
&\stackrel{w}{\longrightarrow}& \, \gamma_i
    &\mbox{weakly in}\quad L^2(\Omega), \quad i=1,\ldots,n,
\\[2mm]
\widetilde{\kappa_m(u_\varepsilon)}
&\stackrel{w}{\longrightarrow}& \, \zeta_m &\mbox{weakly in}\quad
L^2(\Omega), \quad  m=1, 2,
\end{array} \right\}
\quad \text{as} \quad \varepsilon \to 0,
\end{equation}
where $v_0,$  \ $\gamma_i, \, i=1,\ldots,n,$ \ $\zeta_m, \, m=1,
2,$ are some functions which will be determined in what follows.

{\bf 2.} Obviously the $\varepsilon$-periodic functions $T_l\left(
\frac{\cdot}{\varepsilon} \right), \
l=1,\ldots,n,$ defined in (\ref{cell-problem}) satisfy the
following relations
$$
\left\{
   \begin{array}{l}
\partial_{x_i} \left(  a_{ij}(\xi) \partial_{\xi_j}
T_l(\xi)|_{\xi=\frac{x}{\varepsilon}} \right) +
\partial_{x_i} a_{il}^{\varepsilon}(x)=0 \quad \forall\, x\in  \Omega_\varepsilon,
\\[3.0\jot]
\bigl(a_{ij}(\xi)\, \partial_{\xi_j}
T_l(\xi)\, \nu_i(\xi) + a_{il}(\xi) \nu_i(\xi)\bigr)\bigl|_{\xi=\frac{x}{\varepsilon}} =0 \quad \forall\, x\in \Xi_\varepsilon.
   \end{array}
\right.
$$
Multiplying the first relation by $u_\varepsilon\, \phi,$ where
$\phi$ is arbitrary function from $C^\infty_0(\Omega),$ and
integrating over $\Omega_\varepsilon,$ we obtain
\begin{equation}\label{3.1}
\int_{\Omega_\varepsilon}\bigl(a_{ij}(\xi)\, \partial_{\xi_j}
T_l(\xi) + a_{il}(\xi) \bigr)|_{\xi=\frac{x}{\varepsilon}}\,
\bigl(u_\varepsilon \,\partial_{x_i}\phi +
\phi\,\partial_{x_i}u_\varepsilon\bigr)\,dx=0,\quad \ \
l=\overline{1,n}.
\end{equation}

Put the following test-function $\varphi(x)=\varepsilon
T_l(\frac{x}{\varepsilon}) \phi(x), \ x\in \Omega_\varepsilon,$
into the integral identity (\ref{inttot}). The result is as
follows
\begin{multline}\label{3.2}
\int_{\Omega_\varepsilon}a_{ij}^\varepsilon(x)
\partial_{x_j}u_\varepsilon
\partial_{\xi_i} T_l(\xi)|_{\xi=\frac{x}{\varepsilon}}\,\phi(x)\,
dx + \varepsilon \int_{\Omega_\varepsilon}a_{ij}^\varepsilon(x)
\partial_{x_j}u_\varepsilon
\, T_l(\frac{x}{\varepsilon})\,\partial_{x_i}\phi(x)\, dx +
    \\
+ \varepsilon^2\sum_{m=1}^2\int_{\Xi^{(m)}_\varepsilon}
\kappa_m(u_\varepsilon)\, T_l \,\phi\,ds_x =\varepsilon
\int_{\Omega_\varepsilon} f_\varepsilon\, T_l\,\phi\, dx +
\varepsilon^2 \sum_{m=1}^2\int_{\Xi^{(m)}_\varepsilon}
g^{(m)}_\varepsilon \, T_l\,\phi\, ds_x.
\end{multline}
Using (\ref{cond-1}), (\ref{cond-2}) and the identities
(\ref{identities}), it follows from (\ref{3.2}) that
\begin{equation}\label{3.3}
\int_{\Omega_\varepsilon}a_{ij}^\varepsilon(x)
\partial_{x_j}u_\varepsilon
\partial_{\xi_i} T_l(\xi)|_{\xi=\frac{x}{\varepsilon}}\,\phi(x)\,
dx = {\cal O}(\varepsilon) \quad \text{as}\quad \varepsilon\to 0,
\quad  l=\overline{1,n}.
\end{equation}
Subtracting (\ref{3.2}) from (\ref{3.1}), we get
\begin{equation}\label{3.4}
\int_{\Omega}\bigl(a_{ij}(\xi)\, \partial_{\xi_j} T_l(\xi) +
a_{il}(\xi) \bigr)|_{\xi=\frac{x}{\varepsilon}}\,
\widetilde{u_\varepsilon} \,\partial_{x_i}\phi \,dx
 + \int_{\Omega} \widetilde{a_{il}^\varepsilon\, \partial_{x_i}u_\varepsilon}\,
\phi \,dx = {\cal O}(\varepsilon),\quad \ \ l=\overline{1,n}.
\end{equation}
In (\ref{3.4}) we regard that the functions $a_{ij}\,
\partial_{\xi_j}T_l + a_{il}, \ l=1,\ldots,n,$ are equal to zero on $B.$

Let us find the limit of the first summand in the left-hand side
of (\ref{3.4}). At first we note that the limit function $v_0$ in (\ref{3.0}) belongs to $H^1_0(\Omega)$ because of the conectedness
of the domain $\mathbb R^n \setminus \overline{\bigl({\cal B}^{(1)}\cup {\cal B}^{(2)}\bigr)}$ (see \cite{Zhikov-1}-\cite{Zhikov-Rychaho}). Since
$\bigl(a_{ij}(\xi)\, \partial_{\xi_j}T_l(\xi) +
a_{il}(\xi)\bigr)\nu_i(\xi)=0$ at $\xi\in S$  and  the
vector-functions
\begin{equation}\label{solenoidal}
{\bf F}_l=\bigl( a_{1j}(\xi)\, \partial_{\xi_j}T_l(\xi) +
a_{1l}(\xi)\, , \ldots ,\, a_{nj}(\xi)\, \partial_{\xi_j}T_l(\xi)
+ a_{nl}(\xi) \bigr), \quad l=1,\ldots,n,
\end{equation}
are solenoidal in $Q_0$ (see (\ref{cell-problem})), their
zero-extensions into $\square\setminus Q_0$ are also solenoidal in
weak sense, i.e.,
\[
\int_{Q_0} {\bf F}_l(\xi) \cdot \nabla\psi(\xi)\, d\xi =
\int_\square {\bf F}_l(\xi) \cdot \nabla\psi(\xi)\, d\xi = 0 \quad
\forall \psi\in C^\infty_{\rm per}(\square), \quad l=1,\ldots,n.
\]
Then using results by V.V.~Zhikov (see \cite[Th.~2.1]{Zhikov-1}), we get that
\[
\lim_{\varepsilon\to 0} \int_{\Omega}\bigl(a_{ij}(\xi)\,
\partial_{\xi_j} T_l(\xi) + a_{il}(\xi)
\bigr)|_{\xi=\frac{x}{\varepsilon}}\, \widetilde{u_\varepsilon}
\,\partial_{x_i}\phi \,dx = \int_{\Omega}\widehat{a}_{il}\, v_0
\,\partial_{x_i}\phi \,dx.
\]
As a results, it follows from (\ref{3.4}) in the limit passage as
$\varepsilon\to0$ that
\[
\int_{\Omega}\widehat{a}_{il}\, v_0 \,\partial_{x_i}\phi \,dx +
\int_\Omega \gamma_l\, \phi\, dx=0 \quad \forall \phi\in
C^\infty_0(\Omega), \quad (l=1,\ldots,n),
\]
i.e.,
\begin{equation}\label{3.5}
\gamma_l(x)=  \widehat{a}_{il}\, \partial_{x_i}v_0(x) \quad  \mbox{for a.e.} \ x\in
\Omega \quad (l=1,\ldots,n).
\end{equation}

{\bf 4.} Using the extension by zero and the identities
(\ref{identities}), we rewrite the integral identity
(\ref{inttot}) in the following way
\begin{multline}\label{new-inttot}
\int_{\Omega} \widetilde{a_{ij}^\varepsilon \,
\partial_{x_j}u_\varepsilon } \, \partial_{x_i} \varphi\,dx + 
\\ 
+ \sum_{m=1}^2\Biggl(\, \underbrace{\varepsilon
\int_{\Omega_\varepsilon} a_{ij}^\varepsilon(x) \,
\partial_{\xi_j} \psi^{(m)}_0(\xi)|_{\xi=\frac{x}{\varepsilon}} \,
\bigl( \kappa'_m(u_\varepsilon) \,\partial_{x_i}u_\varepsilon
\,\varphi + \kappa_m(u_\varepsilon) \,\partial_{x_i}\varphi\bigr)
\,dx} +
q_m \int_{\Omega} \widetilde{\kappa_m(u_\varepsilon)} \, \varphi\,
dx \Biggr) = 
\\ 
= \int_{\Omega} \chi_{Q_0}^\varepsilon\, f_\varepsilon \varphi\, dx + 
\sum_{m=1}^2\Biggl( \, \underbrace{\varepsilon
\int_{\Omega_\varepsilon} a_{ij}^\varepsilon(x) \,
\partial_{\xi_j} \psi^{(m)}_0(\xi)|_{\xi=\frac{x}{\varepsilon}} \,
\bigl( \partial_{x_i}g^{(m)}_\varepsilon\, \varphi +
g^{(m)}_\varepsilon \,\partial_{x_i}\varphi\bigr) \,dx}  +
\\
+ q_m \int_{\Omega} \chi_{Q_0}^\varepsilon\, g^{(m)}_\varepsilon
\, \varphi\, dx \Biggr)\quad \forall\, \varphi\in
C^\infty_0(\Omega).
\end{multline}

It is easy to see that the pointed summands in (\ref{new-inttot})
vanish as $\varepsilon\to 0;$ the first one due to (\ref{cond-2}),
(\ref{est-0}) and (\ref{uniform estimate}), the second one due to
(\ref{est-0}) and (\ref{cond-1}).

Taking into account (\ref{3.0}), (\ref{3.5}) and (\ref{cond-1}),
we pass to the limit in (\ref{new-inttot}) as $\varepsilon\to0.$
As a result we get the  identity
\begin{equation}\label{new-inttot-1}
\int_{\Omega} \widehat{a}_{ij} \,
\partial_{x_j}v_0 \, \partial_{x_i} \varphi\,dx
+ \sum_{m=1}^2 q_m \int_{\Omega} \zeta_m \, \varphi\, dx = |Q_0|
\int_{\Omega}  f_0 \varphi\, dx + \sum_{m=1}^2 |S^{(m)}|
\int_{\Omega} g^{(m)}_0 \, \varphi\, dx
\end{equation}
for any function $\varphi\in C^\infty_0(\Omega).$ Since the space
$C^\infty_0(\Omega)$ is dense in $H^1_0(\Omega),$ identity
(\ref{new-inttot-1}) is valid for any function $\varphi\in
H^1_0(\Omega).$

{\bf 5.} With the help of (\ref{cond-1}), (\ref{inttot}) and
(\ref{new-inttot-1}) we can find that
\begin{multline}\label{3.13}
\lim_{\varepsilon\to0} E_\varepsilon(u_\varepsilon) =
\lim_{\varepsilon\to 0} \Biggl(
\int_{\Omega_\varepsilon}f_\varepsilon\, u_\varepsilon\,dx +
\sum_{m=1}^2 \int_{\Xi^{(m)}_\varepsilon} g^{(m)}_\varepsilon \,
u_\varepsilon\, ds_x\Biggr) =
\lim_{\varepsilon\to 0} \Biggl( \int_{\Omega}f_\varepsilon\,
\widetilde{u_\varepsilon}\,dx + 
\\
+ \sum_{m=1}^2 \Bigl(\varepsilon
\int_{\Omega_\varepsilon} a_{ij}^\varepsilon(x) \,
\partial_{\xi_j} \psi^{(m)}_0|_{\xi=\frac{x}{\varepsilon}} \,
\bigl( \partial_{x_i}g^{(m)}_\varepsilon\, u_\varepsilon +
g^{(m)}_\varepsilon \,\partial_{x_i}u_\varepsilon\bigr) \,dx + q_m
\int_{\Omega} g^{(m)}_\varepsilon \,
\widetilde{u_\varepsilon} \, dx \Bigr) \Biggr)=
\\
=|Q_0| \int_{\Omega}f_0 v_0\,dx + \sum_{m=1}^2 |S^{(m)}|
\int_{\Omega} g^{(m)}_0  v_0\, dx =  
\int_{\Omega} \widehat{a}_{ij} \partial_{x_j}v_0  \partial_{x_i}v_0\,dx
+ \sum_{m=1}^2 q_m \int_{\Omega} \zeta_m(x) v_0\, dx.
\end{multline}

{\bf 6.} Now it remains to determine the last summand in
(\ref{3.13}). For this we will use the method of Browder and
Minty, a remarkable technique which somehow applies to the
corresponding inequality of monotonicity to justify passing to a
weak limit within a nonlinearity.

Thanks to (\ref{elliptic-cond}) and (\ref{cond-2}), the inequality
of monotonicity in our case reads as follows
\begin{multline}\label{mono-inequality}
\int_{\Omega_\varepsilon} a_{ij}^\varepsilon \,
\partial_{x_j}\bigl(u_\varepsilon - \varphi - \varepsilon T_p \,\partial_{x_p}\varphi \bigr)
 \, \partial_{x_i}\bigl(u_\varepsilon - \varphi - \varepsilon T_q\, \partial_{x_q}\varphi \bigr)
\,dx +
\\
+ \varepsilon \sum_{m=1}^2 \int_{\Xi^{(m)}_\varepsilon} \bigl(
\kappa_m(u_\varepsilon) - \kappa_m(\varphi)\bigr)
\bigl(u_\varepsilon - \varphi\bigr)\, ds_x \ge 0 \qquad \forall\,
\varphi\in C^\infty_0(\Omega),
\end{multline}
which is equivalent to
\begin{multline}\label{mono-inequality-1}
\int_{\Omega_\varepsilon} a_{ij}^\varepsilon \,
\partial_{x_j}u_\varepsilon\, \partial_{x_i}u_\varepsilon\,dx +
\varepsilon \sum_{m=1}^2 \int_{\Xi^{(m)}_\varepsilon}
\kappa_m(u_\varepsilon) \, u_\varepsilon\,ds_x +
\\
+ \int_{\Omega_\varepsilon} a_{ij}^\varepsilon \, \bigl(
\partial_{x_j}\varphi +  \partial_{\xi_j}T_p\,
\partial_{x_p}\varphi \bigr)  \, \bigl(\partial_{x_i}\varphi + 
\partial_{\xi_i}T_q\, \partial_{x_q}\varphi \bigr)\,dx -
 2\int_{\Omega} \widetilde{a_{ij}^\varepsilon \,
\partial_{x_j}u_\varepsilon} \, \partial_{x_i}\varphi\,dx -
\\
- 2\int_{\Omega_\varepsilon} a_{ij}^\varepsilon \,
\partial_{x_j}u_\varepsilon\, \partial_{\xi_i}T_q\, \partial_{x_q}\varphi \,dx -
 2 \varepsilon \int_{\Omega_\varepsilon} a_{ij}^\varepsilon \,
\bigl( \partial_{x_j}u_\varepsilon - \partial_{x_j}\varphi -
\partial_{\xi_j}T_p \,\partial_{x_p}\varphi \bigr)
\, T_q\, \partial^2_{x_i x_q}\varphi \,dx +  
\\
+ \varepsilon^2
\int_{\Omega_\varepsilon} a_{ij}^\varepsilon \, T_p\, T_q\,
\partial^2_{x_j x_p}\varphi\,\partial^2_{x_i x_q}\varphi\, dx -
\\
- \varepsilon \sum_{m=1}^2 \int_{\Xi^{(m)}_\varepsilon} \bigl(
\kappa_m(\varphi)\, u_\varepsilon +
\kappa_m(u_\varepsilon)\,\varphi - \kappa_m(\varphi)\,\varphi
\bigr)\, ds_x \ge 0 \quad \forall\, \varphi\in C^\infty_0(\Omega).
\end{multline}
The limit of the first line in~(\ref{mono-inequality-1}) is equal
to the right-hand side in (\ref{3.13}). The first integral in the
second line can be re-written in the form
\begin{equation}\label{expression-0}
\int_{\Omega_\varepsilon} \Bigl( a_{ij}(\xi) \,
\partial_{\xi_j}\bigl(\xi_p +  T_p \bigr)
\partial_{\xi_i}\bigl(\xi_q +  T_q
\bigr)\Bigr)|_{\xi=\frac{x}{\varepsilon}}
 \, \partial_{x_p}\varphi\, \partial_{x_q}\varphi\,dx.
\end{equation}
It follows from \cite{Zhikov-1} that its limit equals $
\int_{\Omega}\widehat{a}_{pq}\, \partial_{x_p}\varphi\,
\partial_{x_q}\varphi\,dx.$ Due to (\ref{3.3}) the integral in
third line vanishes. Obviously, the limits of summands in the
fourth line are equal to zero. The limits of the integrals in the
last line can be found with the help of
Lemma~\ref{convergence-lemma}. As a results we have
\begin{equation}\label{3.16}
\int_{\Omega} \widehat{a}_{ij}\,
\partial_{x_j}(v_0 - \varphi)\,  \partial_{x_i}(v_0 - \varphi) \,dx
+ \sum_{m=1}^2 q_m \int_{\Omega} \bigl(\zeta_m - |Q_0|
\kappa_m(\varphi)\bigr)\, \bigl( v_0 - \varphi \bigr)\, dx \ge 0.
\end{equation}
Evidently, this inequality holds for any function $\varphi\in
H^1_0(\Omega).$

Fix any $\psi\in C^\infty_0(\Omega)$ and set $\varphi:=v_0 -
\lambda \psi \ (\lambda>0)$ in (\ref{3.16}). We get then
\begin{equation*}
\lambda \int_{\Omega} \widehat{a}_{ij}\,
\partial_{x_j}\psi\,  \partial_{x_i}\psi \,dx
+ \sum_{m=1}^2 q_m \int_{\Omega} \bigl(\zeta_m - |Q_0|
\kappa_m(v_0 - \lambda \psi)\bigr)\, \psi\, dx \ge 0 \quad
\forall\, \psi\in C^\infty_0(\Omega).
\end{equation*}
In the limit (as $\lambda\to0)$ we obtain
\[
 \int_{\Omega} \sum_{m=1}^2 q_m \bigl(\zeta_m - |Q_0|
\kappa_m(v_0)\bigr)\, \psi\, dx \ge 0.
\]
Replacing $\psi$ by $-\psi,$ we deduce that in fact quality holds
above. Thus
\begin{equation}\label{zeta}
\sum_{m=1}^2 q_m \zeta_m(x) = \sum_{m=1}^2 |S^{(m)}|\,
\kappa_m(v_0(x)) \quad  \mbox{for a.e.} \ x\in \Omega.
\end{equation}

{\bf 7.} Returning to (\ref{new-inttot-1}), we see that the
function $v_0$ satisfies the following integral identity
\begin{equation}\label{3.17}
\int_{\Omega} \widehat{a}_{ij} \,
\partial_{x_j}v_0 \, \partial_{x_i} \varphi\,dx
+ \sum_{m=1}^2 |S^{(m)}| \int_{\Omega} \kappa_m(v_0) \, \varphi\,
dx = |Q_0| \int_{\Omega}  f_0 \varphi\, dx + \sum_{m=1}^2
|S^{(m)}| \int_{\Omega} g^{(m)}_0 \, \varphi\, dx
\end{equation}
for any function $\varphi\in H^1_0(\Omega).$ Hence $v_0$ is a weak
solution to the limit problem (\ref{limit-problem}). Thanks to
(\ref{cond-2}) this solution is unique.

Due to the uniqueness of the solution to problem
(\ref{limit-problem}), the above argumentations hold for any
subsequence of $\{\varepsilon\}$ chosen at the beginning of the
proof. By replacing (\ref{zeta}) in (\ref{3.13}), one obtains the
convergence of energies (\ref{energy}).
\end{proof}

%%%%%%%%%%%%%%%%%%%%%%%%%%%%%%%%%%%%%%%%%%%%%%%%%%%%%%%%%%%%%%%%%%%%%

\section{Asymptotic approximation to the solution\\ and the energy integral}

%%%%%%%%%%%%%%%%%%%%%%%%%%%%%%%%%%%%%%%%%%%%%%%%%%%%%%%%%%%%%%%%%%%%%

We take the following approximation
\begin{equation}\label{approx}
\overline{u}_\varepsilon := v_0(x) + \varepsilon T_k \left(
\tfrac{x}{\varepsilon} \right) \partial_{x_k} v_0(x)
\end{equation}
to the solution $u_\varepsilon.$ Substituting the difference
$u_\varepsilon - \overline{u}_\varepsilon,$ we find the residuals
both in the differential equation and boundary conditions.
Straightforward calculation show that
\begin{gather}\label{residual-1}
 -{\cal L}_\varepsilon
\left( u_\varepsilon - \overline{u}_\varepsilon\right) =
f_\varepsilon(x) - f_0(x) - \sum_{m=1}^2 q_m \,\bigl( g^{(m)}_0(x) -
 \kappa_m(v_0(x)) \bigr) + \notag
\\
+\left( a_{ij}(\xi) + a_{ik}(\xi) \partial_{\xi_k} T_j(\xi) -
\tfrac{1}{|Q_0|}\, \widehat{a}_{ij} \right)\left|_{\xi=\frac{x}{\varepsilon}}
\right. \,
\partial^2_{x_i x_j} v_0 +
\varepsilon \partial_{x_i} \left (F_i^\varepsilon(x) \right),
\quad x\in \Omega_\varepsilon;
\end{gather}
\begin{equation}\label{residual-2}
\sigma_\varepsilon\left ( u_\varepsilon - \overline{u}_\varepsilon
\right) = -\varepsilon \kappa_m(u_\varepsilon) + \varepsilon
g^{(m)}_\varepsilon (x) - F_i^\varepsilon(x)\, \nu_i, \quad x\in
\Xi^{(m)}_\varepsilon \ \ (m=1, 2),
\end{equation}
where\  $F_i^\varepsilon(x)= a_{ij}(\frac{x}{\varepsilon})
T_k(\frac{x}{\varepsilon}) \, \partial_{x_j x_k}^2 v_0(x),$ $i=1,\ldots,n,$ and
\begin{equation}\label{residual-3}
\left. \left ( u_\varepsilon - \overline{u}_\varepsilon \right)
\right|_{\Gamma_{\varepsilon}} =-\varepsilon T_k \left( \tfrac{x}{\varepsilon}\right) \partial_{x_k} v_0(x).
\end{equation}

Let $\varphi_\varepsilon$ be a smooth function in $\overline{\Omega}$ such that
$0 \leq \varphi_\varepsilon \leq 1,$
\ $\varphi_\varepsilon(x)=1$  \ if \ $dist(x,\partial \Omega)\leq \varepsilon,$
and $\varphi_\varepsilon(x)=0$  \ if \ $dist(x,\partial \Omega)\geq 2\varepsilon.$
Obviously,
\begin{equation}\label{est-grad}
|\nabla_x \varphi_\varepsilon|\leq c\, \varepsilon^{-1} \quad \text{in} \ \overline{\Omega}.
\end{equation}
With the help of $\varphi_\varepsilon$ we define the following functions
\[
\psi_\varepsilon (x) = - \varepsilon \varphi_{\varepsilon}(x)\, T_k\left(\tfrac{x}{\varepsilon}\right)
\, \partial_{x_k} v_0(x) \quad \text{and}\quad
w_\varepsilon(x) = u_\varepsilon(x) - \overline{u}_\varepsilon(x) - \psi_\varepsilon (x), \quad
x\in \overline{\Omega}_\varepsilon.
\]
It is easy to verify that ${supp}\,(\psi_\varepsilon)\subset
{\cal U}_{2\varepsilon} = \{x\in \overline{\Omega}_\varepsilon: \ \
\text{dist}(x,\partial \Omega)\leq 2\varepsilon \}$
and  $w_\varepsilon$ is a solution to the following problem
\begin{equation*}
\left\{
   \begin{array}{l}
- {\cal L}_\varepsilon(w_\varepsilon) =
f_\varepsilon - f_0 - \sum_{m=1}^2 q_m \, \bigl( g^{(m)}_0(x) -
 \kappa_m(v_0)\bigr) + \varepsilon \partial_{x_i} \left (F_i^\varepsilon(x) \right)
+ {\cal L}_\varepsilon(\psi_\varepsilon) +
\\[2mm]
+\left( a_{ij}(\xi) + a_{ik}(\xi) \partial_{\xi_k} T_j(\xi) -
|Q_0|^{-1} \hat{a}_{ij} \right)\left|_{\xi=\frac{x}{\varepsilon}}
\right. \,
\partial^2_{x_i x_j} v_0
\qquad \text{in} \ \ \Omega_\varepsilon;
\\[2mm]
\sigma_\varepsilon(w_\varepsilon)=
-\varepsilon \kappa_m(u_\varepsilon) + \varepsilon
g^{(m)}_\varepsilon (x) - F_i^\varepsilon(x)\, \nu_i
 - \sigma_\varepsilon(\psi_\varepsilon) \qquad \text{on} \ \
\Xi^{(m)}_\varepsilon \ \ (m=1, 2);
\\[2mm]
w_\varepsilon=0 \qquad \textrm{on} \ \ \Gamma_\varepsilon.
\end{array}
\right.
\end{equation*}

Multiplying the equation of this problem by $w_\varepsilon,$ then integrating by parts
and subtracting identities (\ref{identities}) for $\varphi_m=\kappa_m(v_0)\, w_\varepsilon, \ m=1, 2,$
we get
\begin{multline}\label{int-residual}
\varepsilon \sum_{m=1}^2\int_{\Xi^{(m)}_\varepsilon} \bigl(
\kappa_m(u_\varepsilon) - \kappa_m(v_0)
\bigr)\, w_\varepsilon  \, ds_x +
\int_{\Omega_\varepsilon} a^\varepsilon_{ij}\,  \partial_{x_j} w_\varepsilon
 \, \partial_{x_i}w_\varepsilon\, dx =
\\
=\int_{\Omega_\varepsilon} (f_\varepsilon-f_0) \varphi\, dx
+
\sum_{m=1}^2\Bigl( \varepsilon \int_{\Xi^{(m)}_\varepsilon}g^{(m)}_\varepsilon \, w_\varepsilon
 \, ds_x - q_m \int_{\Omega_\varepsilon}  g^{(m)}_0 \, w_\varepsilon\, dx\Bigr) -
\\
- \varepsilon \sum_{m=1}^2
 \int_{\Omega_\varepsilon} a^\varepsilon_{ij}(x)\,
     \partial_{\xi_j} \psi_0|_{\xi=\frac{x}{\varepsilon}}
\partial_{x_i} \left ( \kappa_m(v_0) \, w_\varepsilon \right)\, dx +
\\
+ \int_{\Omega_\varepsilon}
\left( a_{ij}(\xi) + a_{ik}(\xi) \partial_{\xi_k} T_j(\xi) -
|Q_0|^{-1} \widehat{a}_{ij} \right)\left|_{\xi=\frac{x}{\varepsilon}}
\right. \, \partial^2_{x_i x_j} v_0\, w_\varepsilon\, dx +
\\
+
\varepsilon \int_{\Omega_\varepsilon}
F_i^\varepsilon \partial_{x_i}w_\varepsilon\, dx -
\int_{\Omega_\varepsilon} a^\varepsilon_{ij}\,  \partial_{x_j} \psi_\varepsilon
 \, \partial_{x_i}w_\varepsilon\, dx.
\end{multline}

Due to (\ref{elliptic-cond}), (\ref{cond-2}) and (\ref{equiv-norms}) the left-hand side of
(\ref{int-residual})
is estimated by the following way
\begin{multline}\label{est-bellow}
\int_{\Omega_\varepsilon} a_{ij}^\varepsilon\,
\partial_{x_j} w_\varepsilon \, \partial_{x_i} w_\varepsilon
\, dx + \varepsilon \sum_{m=1}^2\int_{\Xi^{(m)}_\varepsilon}
\left ( \kappa_m (u_\varepsilon)-\kappa_m(v_0)\right)
w_\varepsilon \, ds_x \ge
\\
\ge c_1 \Bigl( \int_{\Omega_\varepsilon} |\nabla w_\varepsilon|^2\, dx
+ \varepsilon \int_{\Xi_\varepsilon} w^2_\varepsilon \, ds_x\Bigr)
- c_2 \varepsilon \int_{\Xi_\varepsilon}
\bigl| \bigr(\varepsilon T_k\, \partial_{x_k}v_0 + \psi_\varepsilon\bigl) w_\varepsilon \bigr| \,ds_x \ge
\\
\ge c_3 \|w_\varepsilon\|^2_{H^1(\Omega_\varepsilon)} -
c_2 \varepsilon \int_{\, \Xi_\varepsilon}
\bigl| \bigl( \varepsilon T_k\, \partial_{x_k}v_0 + \psi_\varepsilon\bigl) w_\varepsilon \bigr|\,ds_x.
\end{multline}

Now estimate the summands in the right-hand side of (\ref{int-residual}). Evidently,
$|\int_{\Omega_\varepsilon} \left ( f_\varepsilon -f_0\right) w_\varepsilon\, dx| \le
\|f_\varepsilon-f_0\|_{L^2(\Omega_\varepsilon)}\|w_\varepsilon \|_{H^1(\Omega_\varepsilon)}.$
With the help of (\ref{identities}), (\ref{cond-1}) and (\ref{est-0})
we bound the second and third terms:
$$
\left |
\sum_{m=1}^2\Bigl( \varepsilon \int_{\Xi^{(m)}_\varepsilon}g^{(m)}_\varepsilon \, w_\varepsilon
 \, ds_x - q_m \int_{\Omega_\varepsilon}  g^{(m)}_0 \, w_\varepsilon\, dx\Bigr)
\right|
=
\sum_{m=1}^2\Bigl( \varepsilon
\Bigl|
     \int_{\Omega_\varepsilon} a_{ij}^\varepsilon
     \left.\partial_{\xi_j} \psi_0(\xi)
     \right|_{\xi=\frac{x}{\varepsilon}}
     \partial_{x_i} (g^{(m)}_\varepsilon w_\varepsilon)\, dx
\Bigr|+
$$
$$
+ q_m
\Bigl|
     \int_{\Omega_\varepsilon} g^{(m)}_\varepsilon w_\varepsilon\, dx
     -\int_{\Omega_\varepsilon} g^{(m)}_0 w_\varepsilon\, dx
\Bigr| \Bigr)
=c_1 \varepsilon \|w_\varepsilon \|_{H^1(\Omega_\varepsilon)}
+c_2 \sum_{m=1}^2\|g^{(m)}_\varepsilon - g^{(m)}_0 \|_{L^2(\Omega_\varepsilon)}
\|w_\varepsilon \|_{H^1(\Omega_\varepsilon)};
$$
$$
\varepsilon \sum_{m=1}^2 \left| \int_{\Omega_\varepsilon}
     a_{ij}^\varepsilon(x)
     \left. \partial_{\xi_j} \psi_0
     \right|_{\xi=\frac{x}{\varepsilon}}
     \partial_{x_i} \left ( \kappa_m(v_0) w_\varepsilon \right)\, dx
\right| \le
$$
$$
\le \varepsilon
c_3 \int_{\Omega_\varepsilon} |\nabla v_0|\, |w_\varepsilon| \, dx
+ \varepsilon c_4 \sum_{m=1}^2 \int_{\Omega_\varepsilon}
|\kappa_m (v_0)| \, |\nabla w_\varepsilon| \, dx
\le \varepsilon c_5 \, \|w_\varepsilon \|_{H^1(\Omega_\varepsilon)}.
$$

Thank to (\ref{hom-matrix}) and the fact that the vector-functions (\ref{solenoidal})
are weak solenoidal in $\square,$ it follows from Lemma 16.4 (\cite{ChechPiatSham})
that
$$
\left|\int_{\Omega_\varepsilon}
\left( a_{ij}(\xi) + a_{ik}(\xi) \partial_{\xi_k} T_j(\xi) - \tfrac{1}{|Q_0|} \,\widehat{a}_{ij} \right)\left|_{\xi=\frac{x}{\varepsilon}}
\right. \, \partial^2_{x_i x_j} v_0\, w_\varepsilon\, dx \right| \le
\varepsilon c_6 \, \|w_\varepsilon \|_{H^1(\Omega_\varepsilon)}.
$$

It is easy to see that $\varepsilon |\int_{\Omega_\varepsilon}
F_i^\varepsilon \partial_{x_i}w_\varepsilon\, dx| \le
\varepsilon c_6 \, \|w_\varepsilon \|_{H^1(\Omega_\varepsilon)}.$
The last summand in (\ref{int-residual}) is estimated with the help of Lemma 1.5 (\cite{OYS})
and (\ref{est-grad}):
$$
\varepsilon \left|\int_{\Omega_\varepsilon}
a^\varepsilon_{ij}\,  \partial_{x_j} \psi_\varepsilon
 \, \partial_{x_i}w_\varepsilon\, dx\right|
= \varepsilon \left|\int_{{\cal U}_{2\varepsilon}}
a^\varepsilon_{ij}\,  \partial_{x_j} \psi_\varepsilon
 \, \partial_{x_i}w_\varepsilon\, dx\right| \le
$$
$$
\le c_7 \int_{{\cal U}_{2\varepsilon}} |\nabla v_0|\, |\nabla w_\varepsilon|\, dx
+ \varepsilon^2 c_8 \int_{{\cal U}_{2\varepsilon}} |D^2 v_0|\, |\nabla w_\varepsilon|\, dx
\le
$$
\begin{equation}\label{last-1}
\le c_7 \|v_0\|_{H^1({\cal U}_{2\varepsilon})} \|w_\varepsilon\|_{H^1({\cal U}_{2\varepsilon})}
+ \varepsilon^2 c_8 \|v_0\|_{H^2(\Omega)} \|w_\varepsilon\|_{H^1(\Omega_\varepsilon)} \le
c_9 \varepsilon^{\frac{1}{2}} \|v_0\|_{H^2(\Omega)} \|w_\varepsilon\|_{H^1(\Omega_\varepsilon)}.
\end{equation}

It is remain to bound the last term in (\ref{est-bellow}). Thanks to (\ref{identities}) and (\ref{est-0})
we have
\[
\varepsilon^2  \int_{\Xi_\varepsilon}
 | T_k\, \partial_{x_k}v_0 \, w_\varepsilon | \,ds_x \le
2 \varepsilon^2
\int_{\Omega_\varepsilon} \Bigl| a_{ij}^\varepsilon
     \left.\partial_{\xi_j} \psi_0(\xi)
     \right|_{\xi=\frac{x}{\varepsilon}}
     \partial_{x_i} \bigl(T_k\, \partial_{x_k}v_0 \, w_\varepsilon\bigr)\Bigr| \, dx
+
\]
\begin{equation}\label{last-2}
+ c_{10} \, \varepsilon  \int_{\Omega_\varepsilon}\Bigl|
T_k\, \partial_{x_k}v_0 \, w_\varepsilon\Bigr| \, dx  \le
c_{11} \varepsilon \|v_0\|_{H^2(\Omega)} \|w_\varepsilon\|_{H^1(\Omega_\varepsilon)}.
\end{equation}
With the same arguments as in (\ref{last-1}) and (\ref{last-2}) we have
\[
\varepsilon \int_{\Xi_\varepsilon}| \psi_\varepsilon \, w_\varepsilon | \,ds_x \le
c_{12} \varepsilon^{\frac{1}{2}} \|v_0\|_{H^2(\Omega)} \|w_\varepsilon\|_{H^1(\Omega_\varepsilon)}.
\]

Finally, we conclude from (\ref{int-residual}), (\ref{est-bellow}) and estimates obtained above
that
\begin{equation}\label{final-1}
\|w_\varepsilon\|_{H^1(\Omega_\varepsilon)} \le
C_1
\Bigl(
     \varepsilon^{\frac{1}{2}}
     +\|f_\varepsilon-f_0\|_{L^2(\Omega_\varepsilon)}
     + \sum_{m=1}^2\|g^{(m)}_\varepsilon - g^{(m)}_0 \|_{L^2(\Omega_\varepsilon)}
\Bigr).
\end{equation}
Since $\|\psi_\varepsilon\|_{H^1(\Omega_\varepsilon)}$ is bounded above by
$C_2 \varepsilon^{\frac{1}{2}},$ we have from (\ref{final-1}) that
\begin{equation}\label{final-estimate}
\|u_\varepsilon - \overline{u}_\varepsilon\|_{H^1(\Omega_\varepsilon)} \le C_3
\Bigl(
\varepsilon^{\frac{1}{2}}
     +\|f_\varepsilon-f_0\|_{L^2(\Omega_\varepsilon)}
     + \sum_{m=1}^2\|g^{(m)}_\varepsilon - g^{(m)}_0 \|_{L^2(\Omega_\varepsilon)}
\Bigr),
\end{equation}
where the constant $C_3$ is independent of $\varepsilon.$

Thus, we have proved the following result.

\begin{theorem}
Between the solution $u_\varepsilon$ to problem  (\ref{bvpipd}) and the approximation function
(\ref{approx}) the estimate (\ref{final-estimate}) holds.
\end{theorem}

With the help of the approximation function (\ref{approx}) and estimate (\ref{final-estimate})
we can obtain an estimate for the energy integrals.

\begin{corollary}
The following estimate
\begin{equation}\label{energy-estimate}
\bigl| E_\varepsilon(u_\varepsilon) - E_0(v_0)\bigr| \le
C \Bigl(\varepsilon^{\frac{1}{2}} +  \|f_\varepsilon-f_0 \|_{L^2 (\Omega_\varepsilon)}
            +\sum_{m=1}^2  \|g_\varepsilon^{(m)}-g_{0}^{(m)}\|_{L^2 (\Omega_\varepsilon)} \Bigr).
\end{equation}
is satisfied, where  the energy integrals  $E_\varepsilon(u_\varepsilon)$ and $E_0(v_0)$ are defined
in (\ref{energy}).
\end{corollary}

\begin{proof}
By virtue of (\ref{final-estimate}) we have
$$
 \partial_{x_{i}} u_{\varepsilon} =\partial_{x_{i}} v_{0}
 +\left. \partial_{\xi_{i}} T_{k} (\xi)\right|_{\xi=\frac{x}{\varepsilon}}
 \partial_{x_{k}} v_{0}  + r_{i}^{\varepsilon} (x),
$$
where
$$
 \|r_{i}^{\varepsilon} \|_{L^2 (\Omega_\varepsilon)}
 \le
C_1 \Bigl(
      \varepsilon^{\frac{1}{2}}
      +\|f_\varepsilon-f_0 \|_{L^2 (\Omega_\varepsilon)}
      +\sum_{m=1}^2
      \|g_\varepsilon^{(m)}-g_{0}^{(m)}\|_{L^2 (\Omega_\varepsilon)}
 \Bigr).
$$

Then
\begin{equation}\label{expression-1}
\int\limits_{\Omega_{\varepsilon}} a_{ij}^{\varepsilon}
\partial_{x_j} u_\varepsilon \, \partial_{x_i} u_\varepsilon \,dx=
\int\limits_{\Omega_{\varepsilon}}
 a_{ij}^{\varepsilon} \bigl(\partial_{x_{i}} v_{0} +\left. \partial_{\xi_{i}} T_{k} (\xi)
      \right|_{\xi=\frac{x}{\varepsilon}}  \partial_{x_{k}} v_{0} \bigr)
 \bigl( \partial_{x_{j}} v_{0} + \left. \partial_{\xi_{j}} T_{l} (\xi)
       \right|_{\xi=\frac{x}{\varepsilon}} \partial_{ x_{l}} v_{0} \bigr) dx + p_\varepsilon,
\end{equation}
where
$$
  p_\varepsilon = 2 \int_{\Omega_{\varepsilon}}
  a_{ij}^{\varepsilon} \left(\partial_{x_{i}} v_{0} +\left.
            \partial_{\xi_{i}} T_{k} (\xi) \right|_{\xi=\frac{x}{\varepsilon}}
      \partial_{x_{k}} v_{0} \right ) r_{\varepsilon}^{j} \, dx
  +\int_{\Omega_{\varepsilon}} a_{ij}^{\varepsilon} r_{\varepsilon}^{i}
  r_{\varepsilon}^{j}\, dx.
$$
Taking into account the boundedness of $a_{ij}^{\varepsilon}$ and
$\partial_{\xi_{j}} T_{l} (\xi)$ and estimate (\ref{final-estimate}), we get
$$
  |p_{\varepsilon}| \le
  c_{1} \left ( \|v_{0}\|_{H^1 (\Omega)}
       \left (\int_{\Omega_\varepsilon} r_{\varepsilon}^{i}
r_{\varepsilon}^{i} \, dx \right )^{\frac{1}{2}}
       +\int_{\Omega_\varepsilon}r_{\varepsilon}^{i} r_{\varepsilon}^{i} \, dx
  \right )
  \le
$$
\begin{equation}\label{est-6}
  \le c_2
  \left (\varepsilon^{\frac{1}{2}} + \|f_\varepsilon-f_0 \|_{L^2 (\Omega_\varepsilon)}
            +\sum_{m=1}^2  \|g_\varepsilon^{(m)}-g_{0}^{(m)}\|_{L^2 (\Omega_\varepsilon)}
       \right ).
\end{equation}
Due to (\ref{cond-1}) we can regard here that $\|f_\varepsilon-f_0 \|^2_{L^2 (\Omega_\varepsilon)}
\le \|f_\varepsilon-f_0 \|_{L^2 (\Omega_\varepsilon)},$ similar for other summands.

Let us introduce the following functions
$$
  H_{kl} (\xi)\equiv  a_{ij}(\xi)\,
  \partial_{\xi_{i}} \left(T_k (\xi)+\xi_{k}\right)
  \partial_{\xi_{j}} \left(T_l (\xi)+\xi_{l}\right ) \,-\, \tfrac{1}{|Q_0|} \, \widehat{a}_{kl}, \quad
k, l =1,\ldots,n.
$$
After extending the functions
$a_{ij}, \ T_{k} , \  \partial_{\xi_{i}} T_{k}, \ k=1,\ldots,n,$
by zero to $\Box\backslash Q_{0},$ the functions $H_{kl}, \  k, l =1,\ldots,n,$ will be 1-periodic with zero average over $\Box.$

By the same way as we rewrote a summand in~(\ref{mono-inequality-1}) (see (\ref{expression-0})) and using the functions $H_{kl}, \  k, l =1,\ldots,n,$ and (\ref{expression-1}), we obtain
\begin{multline}\label{expression-2}
\int_{\Omega_{\varepsilon}} a_{ij}^{\varepsilon}
\partial_{x_j} u_\varepsilon \, \partial_{x_i} u_\varepsilon \,dx -
\int_{\Omega} \widehat{a}_{ij}\partial_{x_j} v_0 \, \partial_{x_i} v_0\,dx=
\int_{\Omega} \left. H_{kl}(\xi) \right|_{\xi=\frac{x}{\varepsilon}}
  \partial_{x_{k}} v_{0} \, \partial_{x_{l}} v_{0} \, dx +
\\
+ \int_\Omega \bigl( \tfrac{1}{|Q_0|} \chi_{Q_0}(\tfrac{x}{\varepsilon}) - 1\bigr)
\widehat{a}_{ij}\partial_{x_j} v_0 \, \partial_{x_i} v_0\,dx + p_{\varepsilon}
=: I_1 + I_1 + p_\varepsilon,
\end{multline}
where $\chi_{Q_0}$ is the characteristic function defined in (\ref{charach-function}).
The summand $I_1$ can be estimated by the same way as in the proof of Theorem~1.3 (\cite[Ch.~2]{OYS}).
As a result, we have
$ %\begin{equation}\label{est-7}
|I_1| \le c_1 \varepsilon \|v_0\|^2_{H^2(\Omega)}.
$%\end{equation}

To estimate $I_2$ we note that $\int_\Box \bigl( \tfrac{1}{|Q_0|} \chi_{Q_0}(\xi) - 1\bigr) d\xi = 0.$
Therefore, with the help of Lem\-ma~1.1~(\cite{ChechPiatSham}) we get
$%\begin{equation}\label{est-8}
|I_2| \le c_2 \varepsilon \|v_0\|^2_{H^2(\Omega)}.
$%\end{equation}

Summarizing (\ref{est-6}) and estimates for $I_1$ and $I_2,$ from (\ref{expression-2}) we deduce
as follows
\begin{multline}\label{est-9}
\Bigl|\int_{\Omega_{\varepsilon}} a_{ij}^{\varepsilon}
\partial_{x_j} u_\varepsilon \, \partial_{x_i} u_\varepsilon \,dx -
\int_{\Omega} \widehat{a}_{ij}\partial_{x_j} v_0 \, \partial_{x_i} v_0\,dx\Bigr| \le
\\
\le C_1
  \Bigl(\varepsilon^{\frac{1}{2}} + \varepsilon \|f_0 \|^2_{L^2 (\Omega)} +
  \|f_\varepsilon-f_0 \|_{L^2 (\Omega_\varepsilon)}
            +\sum_{m=1}^2  \|g_\varepsilon^{(m)}-g_{0}^{(m)}\|_{L^2 (\Omega_\varepsilon)} \Bigr).
\end{multline}

Now consider the difference
\[
I_3:= \varepsilon \sum_{m=1}^2\int_{\Xi_{\varepsilon}^{(m)}} k_{m} (u_\varepsilon) u_\varepsilon \,ds_{x}
  -\sum_{m=1}^2 |S^{m}| \int_{\Omega} k_{m} (v_{0}) v_{0} \,dx.
\]
With the help of integral identities (\ref{identities}) we re-write it in the form
\begin{multline*}
I_3 =  \varepsilon \sum_{m=1}^2\int_{\Omega_{\varepsilon}}
  a_{ij}^{\varepsilon} \left. \partial_{\xi_{j}} \psi_{0}^{(m)} (\xi) \right |_{\xi=\frac{x}{\varepsilon}}
  \partial_{x_{i}} (k_{m}(u_{\varepsilon})u_{\varepsilon}) \, dx +
\\
+ \sum_{m=1}^2 q_{m} \int_{\Omega_{\varepsilon}}
  k_{m} (u_{\varepsilon}) u_{\varepsilon}\, dx - \sum_{m=1}^2 |S^{m}| \int_{\Omega} k_{m}(v_{0}) v_{0} \,dx.
\end{multline*}
Due (\ref{cond-2}), (\ref{est-0}) and (\ref{uniform estimate}) the first term is not grater then $\varepsilon c_1.$
Since
$$
\Bigl|\int_{\Omega_{\varepsilon}} k_{m} (u_{\varepsilon}) u_{\varepsilon}\, dx
- \int_{\Omega_{\varepsilon}} k_{m}(v_0) v_0\, dx\Bigr| \le c_2 \|u_{\varepsilon} -v_0\|_{L^2(\Omega_\varepsilon)},
$$
it remains to estimate the following difference
\[
\sum_{m=1}^2 \Bigl| q_{m} \int_{\Omega_{\varepsilon}}
  k_{m}(v_0) v_0 \, dx - |S^{m}| \int_{\Omega} k_{m}(v_{0}) v_{0} \,dx \Bigr|
= \sum_{m=1}^2 \Bigl|\int_{\Omega}\bigl(q_{m}\chi_{Q_0}(\tfrac{x}{\varepsilon}) -
|S^{m}|\bigr)  k_{m}(v_{0}) v_{0} \,dx \Bigr|.
\]
Thanks to the equality $\int_\Box \bigl( q_{m} \chi_{Q_0}(\xi) - |S^{m}|\bigr) d\xi = 0$ $(q_m=|S^{m}|/|Q_0|)$
and Lemma~1.1~(\cite{ChechPiatSham}), this difference is bounded by $c_3 \varepsilon \|v_0\|_{H^1(\Omega)}.$
Thus, $|I_3|\le c_4\varepsilon + c_2 \|u_{\varepsilon} -v_0\|_{L^2(\Omega_\varepsilon)}.$

Finally, taking into account the previous estimate, (\ref{est-9}), (\ref{final-estimate}) and noting  that $E_\varepsilon(u_\varepsilon) - E_0(v_0)= I_1+I_2+I_3 + p_\varepsilon,$ we arrive to
(\ref{energy-estimate}).
\end{proof}

\end{document}